\RequirePackage{fix-cm}
\documentclass[smallextended]{svjour3}       
\smartqed  

\newtheorem{thm}{Theorem}[section]
\newtheorem{defini}{Definition}[section]
\newtheorem{rem}{Remark}[section]
\newtheorem{lem}{Lemma}[section]
\newtheorem{prop}{Proposition}[section]
\newtheorem{coro}{Corollary}[section]
\newtheorem{ex}{Example}

\def \R {\mathbb{R} }

\usepackage[bookmarksnumbered, colorlinks, plainpages]{hyperref}
\hypersetup{colorlinks=true,linkcolor=red, anchorcolor=green, citecolor=cyan, urlcolor=red, filecolor=magenta, pdftoolbar=true}
\usepackage{dsfont}

\usepackage{xfrac}
\usepackage{graphicx}
\usepackage{hyperref}
\usepackage{amssymb,tikz}
\usepackage{amsfonts}
\usepackage{amsmath}
\usepackage{enumitem}
\def \R {\mathbb{R} }

\begin{document}

\title{Eigenvalue type problem in $s(.,.)$-fractional Musielak-Sobolev spaces
}
\subtitle{ }


\author{ Mohammed SRATI 
}


\institute{
            Mohammed SRATI \at
               High School of Education and Formation (ESEF), University Mohammed First, Oujda, Morocco. \\
                             \email{srati93@gmail.com  }                             
}

\date{Received: date / Accepted: date}

\maketitle

\begin{abstract}
In this paper, we introduce the $s(.,.)$-fractional Musielak-Sobolev spaces $W^{s(x,y)}L_{\varPhi_{x,y}}(\Omega)$. Then, we show that there exists $\lambda_*>0$ such that any $\lambda\in(0, \lambda_*)$ is an eigenvalue for the following  problem, by means of Ekeland's ariational principle
  $$\label{P}
  (\mathcal{P}_a)  \left\{ 
     \begin{array}{clclc}
  \left( -\Delta\right)^{s(x,.)}_{a_{(x,.)}} u & = & \lambda |u|^{q(x)-2}u  & \text{ in }& \Omega, \\\\
      u & = & 0 \hspace*{0.2cm}  & \text{ in } & \R^N\setminus \Omega,
     \end{array}
     \right. 
  $$
where  $\Omega$ is a bounded open subset of  $\R^N$ with $C^{0,1}$-regularity  and bounded boundary.

\end{abstract}
\section{Introduction}\label{S1}
The theory of fractional modular spaces is well developed in the last few years, particularly in the fractional Orlicz-Sobolev spaces $W^sL_\varPhi(\Omega)$ (see \cite{sr5,3,sr_mo,sal2,sal1}) and in the fractional Sobolev spaces  with variable exponents $W^{s,p(x,y)}(\Omega)$ (see \cite{SH,SRH,SS2020,bah1,bah2,ku,bah5,s2..,s2...}). The study of variational problems where the modular function  satisfies  nonpolynomial growth conditions instead of having the usual $p$-structure arouses much interest in the development of applications to electrorheological fluids as an important class of non-Newtonian fluids (sometimes referred to as smart fluids). The electro-rheological fluids are characterized by their ability to drastically change the mechanical properties under the influence of an external electromagnetic field. A mathematical model of electro-rheological fluids was proposed by Rajagopal and Ruzicka (we refer the reader to \cite{maria1,maria2,e2} for more details).\\
 
On the contrary, amalgamating the functional characteristics of variable exponent Lebesgue spaces and Orlicz spaces leads us to the Musielak-Orlicz spaces. This particular functional structure has been extensively explored since the 1950s, with Nakano \cite{ch5.na} laying the groundwork, followed by further development by Musielak and Orlicz \cite{ch5.mu,mu}.  A logical inquiry arises: can we extend this generalization to the fractional domain? Azroul et al provide an affirmative response to this question in \cite{benkirane,benkirane2,benkirane3}. In essence, the authors have introduced the fractional Musielak-Sobolev space $W^{s}L_{\varPhi_{x,y}}(\Omega)$, which serves as a natural extension of the aforementioned functional spaces. \\
  
  On the other hand in fractional spaces, or fractional Sobolev spaces,  if the order 
  $s$ depends on a variable, i.e, $s=s(.,.)$ then it is referred to as variable fractional derivative order. This situation may arise in contexts where the derivative order depends on the position or other characteristics of the independent variable. Using a variable fractional derivative order adds an additional dimension to mathematical modeling. Here are some points to consider:
  \begin{enumerate}
  \item \textbf{Local Adaptation:}  By allowing 
  $s$ to depend on 
  $x$, one can locally adapt the derivative order based on specific characteristics of the function or the region of space under study. This can be particularly useful for modeling physically variable phenomena in space.
  \item \textbf{Variable Boundary Problems:} The variable derivative order can be used to model non-homogeneous boundary problems where conditions vary depending on the position in space.
 \item \textbf{Applications in Signal Processing:} In the field of signal processing, a variable derivative order can be used to model situations where the regularity of a signal varies in time or space.
 \item \textbf{Modeling Dynamic Phenomena:} If 
 $s(.,.)$ varies with time, this can be used to model dynamic phenomena where the derivative order changes over time.
   \end{enumerate}
 However, it's important to note that introducing a variable derivative order can make mathematical analysis more complex. The resulting equations may be more challenging to solve, and advanced mathematical tools, such as integral equations or adaptive numerical methods, may be necessary. In general, the use of a variable derivative order is motivated by the need to accurately model complex and variable physical phenomena.\\
 
So, in this present work, we introduce the  $s(.,.)$-fractional Musielak-Sobolev space 
$W^{s(x,y)}L_{\varPhi_{x,y}}(\Omega)$, with the order 
  $s$ depends on a variable, i.e, $s=s(.,.)$. Also, we study the existence  of the eigenvalues of problem
involving non-local  $s(.,.)$-order operator  $\left( -\Delta\right)^{s(x,.)}_{a_{(x,.)}}$  with variable exponents $a(.,.)$.\\ 

Here we would like to emphasize that in our work we have considered the variable growth on the order $s$ as well. Consequently, our focus is to study  the following  eigenvalue problem:
 $$\label{P}
 (\mathcal{P}_a)  \left\{ 
    \begin{array}{clclc}
 \left( -\Delta\right)^{s(x,.)}_{a_{(x,.)}} u & = & \lambda |u|^{q(x)-2}u  & \text{ in }& \Omega, \\\\
     u & = & 0 \hspace*{0.2cm}  & \text{ in } & \R^N\setminus \Omega,
    \end{array}
    \right. 
 $$
  where $\Omega$ is an open bounded subset in $\R^N$, $N\geqslant 1$,   with Lipschitz boundary $\partial \Omega$, $q : \overline{\Omega}\rightarrow (1, \infty)$ is bounded continuous function, and  $(-\Delta)^{s(x,.)}_{a_{(x,.)}}$ is the nonlocal   $s(.,.)$-order  operator of elliptic type defined as follows
    {\small  $$
               \begin{aligned}
               (-\Delta)^{s(x,.)}_{a_{(x,.)}}u(x)=2\lim\limits_{\varepsilon\searrow 0} \int_{\R^N\setminus B_\varepsilon(x)} a_{(x,y)}\left( \dfrac{|u(x)-u(y)|}{|x-y|^{s(x,y)} }\right)\dfrac{u(x)-u(y)}{|x-y|^{s(x,y)}} \dfrac{dy}{|x-y|^{N+s(x,y)} }
               \end{aligned}
                $$}  
  for all $x\in \R^N$, where:\\
  $\bullet$ $s(.,.)~ : \overline{\Omega}\times\overline{\Omega}\rightarrow (0,1)$ is a continuous function such that:
  \begin{equation}
  s(x,y)=s(y,x)~~ \forall x,y \in \overline{\Omega}\times\overline{\Omega},
  \end{equation}
  \begin{equation}
  0<s^-=\inf\limits_{\overline{\Omega}\times\overline{\Omega}}s(x,y)\leqslant s^+=\sup\limits_{\overline{\Omega}\times\overline{\Omega}}s(x,y)<1.
  \end{equation}
   $\bullet$ $a_{(x,y)}(t):=a(x,y,t) : \overline{\Omega}\times\overline{\Omega}\times \R\longrightarrow \R$   is symmetric function :
 \begin{equation}\label{n4}
 a(x,y,t)=a(y,x,t) ~~ \forall(x,y,t)\in \overline{\Omega}\times\overline{\Omega}\times \R,\end{equation}
   and the function : $\varphi(.,.,.) : \overline{\Omega}\times\overline{\Omega}\times \R \longrightarrow \R$ defined by  
$$
  \varphi_{x,y}(t):=\varphi(x,y,t)= \left\{ 
          \begin{array}{clclc}
        a(x,y,|t|)t   & \text{ for }& t\neq 0, \\\\
          0  & \text{ for } & t=0,
          \end{array}
          \right. 
$$
is increasing homeomorphism from $\R$ onto itself. Let 
$$\varPhi_{x,y}(t):=\varPhi(x,y,t)=\int_{0}^{t}\varphi_{x,y}(\tau)d\tau~~\text{ for all } (x,y)\in \overline{\Omega}\times\overline{\Omega},~~\text{ and all } t\geqslant 0.$$  
Then, $\varPhi_{x,y}$ is a Musielak function (see \cite{mu}), that is
\begin{itemize}
\item[$\star$]  $\varPhi(x,y,.)$ is a $\varPhi$-function for every $(x,y)\in\overline{\Omega}\times\overline{\Omega}$, i.e.,   is continuous, nondecreasing function with $\varPhi(x,y,0)= 0$, $\varPhi(x,y,t)>0$ for $t>0$ and $\varPhi(x,y,t)\rightarrow \infty$ as $t\rightarrow \infty$.
\item[$\star$] For every $t\geqslant 0$, $\varPhi(.,.,t) : \overline{\Omega}\times\overline{\Omega} \longrightarrow \R$ is a measurable function.
\end{itemize}
Also, we take $ \widehat{a}_x(t):=\widehat{a}(x,t)=a_{(x,x)}(t)  ~~ \forall~ (x,t)\in \overline{\Omega}\times \R$. Then the function $\widehat{\varphi}(.,.) : \overline{\Omega}\times \R \longrightarrow \R$ defined  by : 
  $$
     \widehat{\varphi}_{x}(t):=\widehat{\varphi}(x,t)= \left\{ 
          \begin{array}{clclc}
        \widehat{a}(x,|t|)t   & \text{ for }& t\neq 0, \\\\
          0  & \text{ for } & t=0,
          \end{array}
          \right. 
       $$
is increasing homeomorphism from $\R$ onto itself. If we set 
\begin{equation}\label{phi}
\widehat{\varPhi}_{x}(t):=\widehat{\varPhi}(x,t)=\int_{0}^{t}\widehat{\varphi}_{x}(\tau)d\tau ~~\text{ for all}~~ t\geqslant 0.
\end{equation}  
Then, $\widehat{\varPhi}_{x}$ is a Musielak function.\\

The results of this work present a generalization to several situations. Note that, when we take $a_{x,y}(t)=|t|^{p(x,y)-2}$ where $p:\overline{\Omega}\times\overline{\Omega}\longrightarrow(1,+\infty)$ is a continuous bounded function, then our nonlocal operator $(-\Delta)^{s(x,.)}_{a_{(x,.)}}$  which can be seen as
a generalization of the nonlocal  $s(.,.)$-order  operator with variable exponent $(-\Delta)^{s(x,.)}_{p(x,.)}$ (see \cite{s1}) defined as
$$(-\Delta)_{p(x,.)}^{s(x,.)}u(x)=2\lim\limits_{\varepsilon\searrow 0} \int_{\R^N\setminus B_\varepsilon(x)}\frac{|u(x)-u(y)|^{p(x,y)-2}(u(x)-u(y))}{|x-y|^{N+s(x,y)p(x,y)}}~dy~~~~~~$$
 for all $x\in \R^N$, (see also   \cite{bah4,s2.,s2}.\\
 
Moreover, this work brings us back to introduce the  $s(.,.)$-fractional $a$-Laplacian $(-\Delta)_{a}^{s(x,.)}$ if $a_{x,y}(t) = a(t)$, i.e. the function $a$ is independent of variables $x, y$. Then, we obtain the following nonlocal non-local  $s(.,.)$-order operator $(-\Delta)_{a}^{s(x,.)}$, defined as
 {\small  $$
               \begin{aligned}
               (-\Delta)^{s(x,.)}_{a}u(x)=2\lim\limits_{\varepsilon\searrow 0} \int_{\R^N\setminus B_\varepsilon(x)} a\left( \dfrac{|u(x)-u(y)|}{|x-y|^{s(x,y)} }\right)\dfrac{u(x)-u(y)}{|x-y|^{s(x,y)}} \dfrac{dy}{|x-y|^{N+s(x,y)} }
               \end{aligned}
                $$} 
 for all $x\in \R^N$.
 
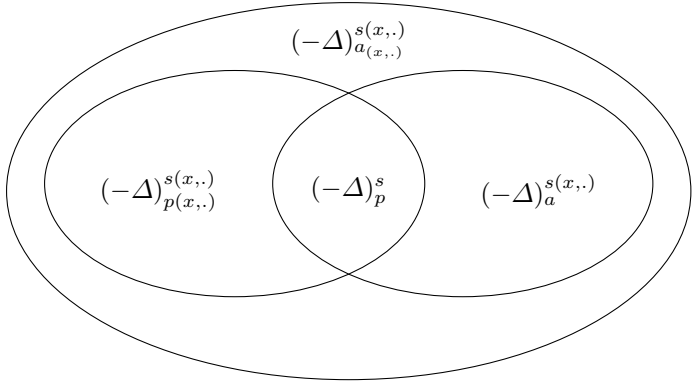
 \begin{center}
   	\def\thirdellipse{(-1.5,.6) ellipse (2.5 and 1.5)}
   	\def\fourthellipse{(1.5,.6) ellipse (2.5 and 1.5)}
   	\def\fifthellipse{(0,0.5) ellipse (4.5 and 2.5)}
   	\begin{tikzpicture}  	
   	\draw\thirdellipse;
   	\draw\fourthellipse;
   	\draw\fifthellipse;
   	\node at (0,0.5) {$(-\Delta)_{p}^{s}$};
   	\node at (-2.5,0.5) {$(-\Delta)_{p(x,.)}^{s(x,.)}$};
   	\node at (2.5, 0.5) {$(-\Delta)_{a}^{s(x,.)}$};
   	\node at (0, 2.5) {$(-\Delta)_{a_{(x,.)}}^{s(x,.)}$};
   	\end{tikzpicture}
    \end{center}

Musielak spaces and variable exponent Lebesgue spaces share similarities with respect to the modulation of norms based on weights or exponents that can vary locally. However, the specific innovation of Musielak spaces lies in their ability to allow local weighting by a weight function, adapted to deal with singularities or specific characteristics of a function. This ability is not as direct in Lebesgue spaces with variable exponent.\\

       This paper is organized as follows, In Section \ref{S1}, we  set  the problem  \hyperref[P]{$(\mathcal{P}_{a})$}. Moreover,  we are introduced the new   nonlocal  $s(.,.)$-order operator $(-\Delta)_{a(x,.)}^{s(x,.)}$.  The Section \ref{S2}, is devoted to recall
   some properties of fractional Musielak-Sobolev spaces. In section \ref{S3}, we introduce the $s(.,.)$-fractional Musielak-Sobolev spaces and we establish some qualitative properties of these new spaces. In section \ref{S4},    by means of Ekeland's variational principle,
       we obtain the existence  of $\lambda_*>0$ such that for any $\lambda\in(0,\lambda_*)$, is an eigenvalue for the following  problem  \hyperref[P]{$(\mathcal{P}_{a})$}.   In Section \ref{S5}, we present some examples which illustrate our  results. 
        \section{Preliminaries results}\label{S2}                      
To deal with this situation we define the fractional Musielak-Sobolev space to
investigate Problem \hyperref[P]{$(\mathcal{P}_{a})$}. Let us recall the definitions and some elementary
properties of this spaces. We refer the reader to \cite{benkirane,benkirane2} for further reference
and for some of the proofs of the results in this section.

 For the function $\widehat{\varPhi}_x$ given in (\ref{phi}), we introduce the Musielak space as follows
  $$L_{\widehat{\varPhi}_x} (\Omega)=\left\lbrace u : \Omega \longrightarrow \R \text{ mesurable }: \int_\Omega\widehat{\varPhi}_x(\lambda |u(x)|)dx < \infty \text{ for some } \lambda>0 \right\rbrace. $$
The space $L_{\widehat{\varPhi}_x} (\Omega)$ is a Banach space endowed with the Luxemburg norm 
$$||u||_{\widehat{\varPhi}_x}=\inf\left\lbrace \lambda>0 \text{ : }\int_\Omega\widehat{\varPhi}_x\left( \dfrac{|u(x)|}{\lambda}\right) dx\leqslant 1\right\rbrace. $$
 The conjugate function of $\varPhi_{x,y}$ is defined by $\overline{\varPhi}_{x,y}(t)=\int_{0}^{t}\overline{\varphi}_{x,y}(\tau)d\tau~~\text{ for all } (x,y)\in\overline{\Omega}\times\overline{\Omega},~~ \text{ and all } t\geqslant 0$, where $\overline{\varphi}_{x,y} : \R\longrightarrow \R$ is given by $\overline{\varphi}_{x,y}(t):=\overline{\varphi}(x,y,t)=\sup\left\lbrace s \text{ : } \varphi(x,y,s)\leqslant t\right\rbrace.$     Throughout this paper, we assume that there exist two positive constants $\varphi^+$ and $\varphi^-$ such that 
\begin{equation}\label{v1}\tag{$\varPhi_1$}
    1<\varphi^-\leqslant\dfrac{t\varphi_{x,y}(t)}{\varPhi_{x,y}(t)}\leqslant \varphi^+<+\infty \text{ for all } (x,y)\in\overline{\Omega}\times\overline{\Omega}~~\text{ and all } t\geqslant 0. \end{equation}
    This relation implies  that
    \begin{equation}\label{A2}
        1<\varphi^-\leqslant \dfrac{t\widehat{\varphi}_{x}(t)}{\widehat{\varPhi}_{x}(t)}\leqslant\varphi^+<+\infty\text{ for all } x\in\overline{\Omega}~~\text{ and all } t\geqslant 0.\end{equation}
             It follows that  $\varPhi_{x,y}$ and $\widehat{\varPhi}_{x}$ satisfy the global $\Delta_2$-condition (see \cite{ra}), written $\varPhi_{x,y}\in \Delta_2$ and $\widehat{\varPhi}_{x}\in \Delta_2$, that is,
    \begin{equation}\label{r1}
    \varPhi_{x,y}(2t)\leqslant K_1\varPhi_{x,y}(t)~~ \text{ for all } (x,y)\in\overline{\Omega}\times\overline{\Omega}~~\text{ and  all } t\geqslant 0,
    \end{equation} and
    \begin{equation}\label{rr1}
        \widehat{\varPhi}_{x}(2t)\leqslant K_2\widehat{\varPhi}_{x}(t) ~~\text{ for any } x\in\overline{\Omega}~~\text{ and  all } t\geqslant 0,
        \end{equation}
 where $K_1$ and $K_2$ are two positive constants. 
 
 Furthermore, we assume that $\varPhi_{x,y}$ satisfies the following condition
  \begin{equation}\label{f2.}\tag{$\varPhi_2$}
  \text{ the function } [0, \infty) \ni t\mapsto \varPhi_{x,y}(\sqrt{t}) \text{ is convex. }
  \end{equation}
 
    \begin{defini}
     Let $A_x(t)$, $B_x(t): \R^+\times \Omega\longrightarrow \R^+$ be two Musielak functions. 
            $A_x$ is stronger $($resp essentially stronger$)$ than $B_x$,  $A_x\succ B_x$ (resp $A_x\succ\succ B_x$) in symbols, if for almost every $x\in \overline{\Omega}$ 
          $$B(x,t)\leqslant A( x,a t),~~ t\geqslant t_0\geqslant 0$$
          for some $($resp for each$)$ $a>0$ and $t_0$ (depending on $a$).
    \end{defini}
     \begin{rem}[{\cite[Section 8.5]{1}}]
          $A_x\succ\succ B_x$  is equivalent to the condition \\
                         $$\lim_{t\rightarrow \infty}\left(\sup\limits_{x\in \overline{\Omega}}\dfrac{B( x,\lambda t)}{A(x,t)}\right) =0$$
                         for all $\lambda>0$. 
                            \end{rem}
         
   Now,  we  define the fractional Musielak-Sobolev space as introduce in \cite{benkirane} as follows 
    \begingroup\makeatletter\def\f@size{9}\check@mathfonts$$ W^s{L_{\varPhi_{x,y}}}(\Omega)=\Bigg\{u\in L_{\widehat{\varPhi}_x}(\Omega) :  \int_{\Omega} \int_{\Omega} \varPhi_{x,y}\left( \dfrac{\lambda| u(x)- u(y)|}{|x-y|^s}\right) \dfrac{dxdy}{|x-y|^N}< \infty \text{ for some } \lambda >0 \Bigg\}.
$$\endgroup
This space can be equipped with the norm
\begin{equation}\label{r2}
||u||_{s,\varPhi_{x,y}}=||u||_{\widehat{\varPhi}_x}+[u]_{s,\varPhi_{x,y}},
\end{equation}
where $[.]_{s,\varPhi_{x,y}}$ is the Gagliardo seminorm defined by 
$$[u]_{s,\varPhi_{x,y}}=\inf \Bigg\{\lambda >0 :  \int_{\Omega} \int_{\Omega} \varPhi_{x,y}\left( \dfrac{|u(x)- u(y)|}{\lambda|x-y|^s}\right) \dfrac{dxdy}{|x-y|^N}\leqslant 1 \Bigg\}.
$$

\begin{thm}$($\cite{benkirane}$)$.
       Let $\Omega$ be an open subset of $\R^N$, and let $s\in (0,1)$. The space $W^sL_{\varPhi_{x,y}}(\Omega)$ is a Banach space with respect to the norm $(\ref{r2})$, and a  separable $($resp. reflexive$)$ space if and only if $\varPhi_{x,y} \in \Delta_2$ $($resp. $\varPhi_{x,y}\in \Delta_2 $ and $\overline{\varPhi}_{x,y}\in \Delta_2$$)$. Furthermore,
              if   $\varPhi_{x,y} \in \Delta_2$ and $\varPhi_{x,y}(\sqrt{t})$ is convex, then  the space $W^sL_{\varPhi_{x,y}}(\Omega)$ is an uniformly convex space.\end{thm}

           \begin{defini}$($\cite{benkirane}$)$.
           We say that $\varPhi_{x,y}$ satisfies the fractional boundedness condition, written $\varPhi_{x,y}\in \mathcal{B}_{f}$, if
         \begin{equation}\tag{$\varPhi_3$}
        \label{v3}         
           \sup\limits_{(x,y)\in \overline{\Omega}\times\overline{\Omega}}\varPhi_{x,y}(1)<\infty.  \end{equation}
           \end{defini}
           \begin{thm}  $($\cite{benkirane}$)$.    \label{TT}
                       Let $\Omega$ be an open subset of $\R^N$,  and  $0<s<1$. Assume that  $\varPhi_{x,y}\in \mathcal{B}_{f}$. 
                       Then,
                       $$C^2_0(\Omega)\subset W^sL_{\varPhi_{x,y}}(\Omega).$$
                  \end{thm}
                  
            \begin{lem}$($\cite{benkirane}$)$ $\label{2.2..}$ Assume that \hyperref[v1]{$(\varPhi_1)$} is satisfied. Then the following inequalities hold true:
                              \begin{equation}\label{3.}
            \varPhi_{x,y}(\sigma t)\geqslant \sigma^{\varphi^-}\varPhi_{x,y}(t) ~~\text{ for all } t>0  \text{ and any  } \sigma>1,
                              \end{equation}
                 \begin{equation}\label{3.2}
            \varPhi_{x,y}(\sigma t)\geqslant \sigma^{\varphi^+}\varPhi_{x,y}(t) ~~\text{ for all }  t>0  \text{ and any } \sigma\in (0,1),
                              \end{equation}
             \begin{equation}\label{r10}
                                 \varPhi_{x,y}(\sigma t)\leqslant \sigma^{\varphi^+}\varPhi_{x,y}(t) ~~\text{ for all } t>0 \text{ and any } \sigma>1,
                                 \end{equation} 
                \begin{equation}\label{r11}
                                \varPhi_{x,y}(t)\leqslant \sigma^{\varphi^-}\varPhi_{x,y}\left( \dfrac{t}{\sigma} \right) ~~\text{ for all } t>0 \text{ and any } \sigma \in(0,1).
                                \end{equation}                                  
                              \end{lem}        
       
   
                     For any $u \in W^sL_{\varPhi_{x,y}}(\Omega)$, we define the modular function on  $W^sL_{\varPhi_{x,y}}(\Omega)$  as follows  
                    \begin{equation}\label{modN}
        \varPsi(u)=\displaystyle\int_{\Omega} \int_{\Omega} \varPhi_{x,y}\left( \dfrac{ |u(x)- u(y)|}{|x-y|^s}\right) \dfrac{dxdy}{|x-y|^N}+\int_{\Omega}\widehat{\varPhi}_{x}\left( |u(x)|\right) dx. \end{equation}                           
        \begin{prop}$($\cite{benkirane}$)$.\label{mod}
         Assume that (\ref{v1}) is satisfied. Then for any $u \in W^sL_{\varPhi_{x,y}}(\Omega)$, the following relations hold true:
           \begin{equation}\label{mod1}
     ||u||_{s,\varPhi_{x,y}}>1\Longrightarrow      ||u||_{s,\varPhi_{x,y}}^{\varphi^-} \leqslant  \varPsi(u)\leqslant  ||u||_{s,\varPhi_{x,y}}^{\varphi^+},
           \end{equation}
           \begin{equation}\label{mod2}
                ||u||_{s,\varPhi_{x,y}}<1\Longrightarrow    ||u||_{s,\varPhi_{x,y}}^{\varphi^+} \leqslant  \varPsi(u)\leqslant  ||u||_{s,\varPhi_{x,y}}^{\varphi^-}. \end{equation}
           \end{prop}
           
         We Define a closed linear subspace of $W^sL_{\varPhi_{x,y}}(\Omega)$ as follows
             $$W^s_0L_{\varPhi_{x,y}}(\Omega)=\left\lbrace u\in W^sL_{\varPhi_{x,y}}(\R^N) \text{ : } u=0 \text{ a.e in } \R^N\setminus \Omega \right\rbrace.$$ 
         \begin{thm} $($\cite{benkirane2}$)$ \label{pc}
                                 Let $\Omega$ be a bounded open subset of  $\R^N$ with $C^{0,1}$-regularity 
                                       and bounded boundary,  let $s\in (0,1)$.
                                  Then there exists a positive constant $\gamma$ such that    
                        $$ ||u||_{\widehat{\varPhi}_x}\leqslant \gamma [u]_{s,\varPhi_{x,y}} \text{ for all   }  u \in  W^s_0L_{\varPhi_{x,y}}(\Omega).$$
                                \end{thm}       
  We denote by $\widehat{\varPhi}_{x}^{-1}$ the inverse function of $\widehat{\varPhi}_{x}$ which satisfies the following conditions:
       \begin{equation}\label{15}
       \int_{0}^{1} \dfrac{\widehat{\varPhi}_{x}^{-1}(\tau)}{\tau^{\frac{N+s}{N}}}d\tau<\infty~~ \text{ for all } x\in \overline{\Omega},
       \end{equation}
       
       \begin{equation}\label{16n}
       \int_{1}^{\infty} \dfrac{\widehat{\varPhi}_{x}^{-1}(\tau)}{\tau^{\frac{N+s}{N}}}d\tau=\infty ~~\text{ for all }x\in \overline{\Omega}.
       \end{equation}
      Note that, if $\varphi_{x,y}(t)=|t|^{p(x,y)-1}$, then (\ref{15}) holds precisely when $sp(x,y)<N$ for all $(x,y)\in \overline{\Omega}\times \overline{\Omega}$.\\
       If (\ref{16n}) is satisfied, we define the inverse  Musielak conjugate function of $\widehat{\varPhi}_x$ as follows
       \begin{equation}\label{17}
       (\widehat{\varPhi}^*_{x,s})^{-1}(t)=\int_{0}^{t}\dfrac{\widehat{\varPhi}_{x}^{-1}(\tau)}{\tau^{\frac{N+s}{N}}}d\tau.
       \end{equation}
        \begin{thm}\cite{benkirane2}\label{th2.}
      Let $\Omega$  be a bounded open
       subset of  $\R^N$ with $C^{0,1}$-regularity 
         and bounded boundary. If $(\ref{15})$ and  $(\ref{16n})$  hold, then 
      \begin{equation}\label{18}
       W^s{L_{\varPhi_{x,y}}}(\Omega)\hookrightarrow L_ {\widehat{\varPhi}^*_{x,s}}(\Omega).
      \end{equation}
     Moreover, the embedding
                \begin{equation}\label{27}
                 W^s{L_{\varPhi_{x,y}}}(\Omega)\hookrightarrow L_{B_x}(\Omega),
                \end{equation}
                is compact for all $B_x\prec\prec \widehat{\varPhi}^*_{x,s}$.
                \end{thm}

   Next, we recall some useful properties of variable exponent spaces. For more details we refer the reader to \cite{s23,s27}, and the references therein.\\ 
     Consider the set
      $$C_+(\overline{\Omega})=\left\lbrace q\in C(\overline{\Omega}): q(x)>1 \text{ for all } x \in\overline{\Omega}\right\rbrace .$$
      For all $q\in C_+(\overline{\Omega}) $, we define $$q^{+}= \underset{x\in \overline{\Omega}}{\sup}~q(x) \quad\text{and}\quad q^{-}= \underset{x \in \overline{\Omega}}{\inf}~q(x).$$
   For any  $q\in C_+(\overline{\Omega}) $, we define the variable exponent Lebesgue space as $$L^{q(x)}(\Omega)=\bigg\{u:\Omega\longrightarrow \mathbb{R} ~~\text{measurable}: \int_{\Omega}|u(x)|^{q(x)}dx<+\infty
   \bigg\}.$$
   This vector space endowed with the \textit{Luxemburg norm}, which is defined by
   $$\|u\|_{L^{q(x)}(\Omega)}= \inf \bigg\{\lambda>0:\int_{\Omega}\bigg|\frac{u(x)}{\lambda}\bigg|^{q(x)}dx \leqslant1 \bigg\}$$
   is a separable reflexive Banach space.
   
    A very important role in manipulating the generalized Lebesgue spaces with variable exponent is played by the modular of the $L^{q(x)}(\Omega)$ space, which defined by
    $$
    \begin{aligned}
   \rho_{q(.)}: L^{q(x)}(\Omega)&\longrightarrow\mathbb{R}\\
     u&\longmapsto\rho_{q(.)}(u)=\int_{\Omega}|u(x)|^{q(x)}dx.
    \end{aligned}
    $$
   \begin{prop}\label{2.2}
   Let $u\in  L^{q(x)}(\Omega) $, then we have
   \begin{enumerate}[label=(\roman*)]
   \item $\|u\|_{L^{q(x)}(\Omega)}<1$ $(resp. =1, >1)$ $\Leftrightarrow$ $ \rho_{q(.)}(u)<1$ $(resp. =1, >1)$,
   \item  $\|u\|_{L^{q(x)}(\Omega)}<1$ $\Rightarrow$ $\|u\|^{q{+}}_{L^{q(x)}(\Omega)}\leqslant \rho_{q(.)}(u)\leqslant \|u\|^{q{-}}_{L^{q(x)}(\Omega)}$,
   \item  $\|u\|_{L^{q(x)}(\Omega)}>1$ $\Rightarrow$ $\|u\|^{q{-}}_{L^{q(x)}(\Omega)}\leqslant \rho_{q(.)}(u)\leqslant \|u\|^{q{+}}_{L^{q(x)}(\Omega)}$.
   \end{enumerate}
   \end{prop}
      Finally, the proof of our existence result is based on the following Ekeland's variational principle theorem. 
   \begin{thm}\label{th1}(\cite{ek})
   Let V be a complete metric space and $F : V \longrightarrow \R\cup \left\lbrace +\infty\right\rbrace$ be a lower semicontinuous functional on $V$, that is bounded below and not identically equal to $+\infty$. Fix $\varepsilon>0$ and a  point $u\in V$ 
     such that
    $$F(u)\leqslant \varepsilon +\inf\limits_{x\in V}F(x).$$ Then for every $\gamma > 0$,
     there exists some point $v\in V$ such that :
     $$F(v)\leqslant F(u),$$
     $$d(u,v)\leqslant \gamma,$$
     and for all $w\neq v$
     $$F(w)> F(v)-\dfrac{\varepsilon}{\gamma}d(v,w).$$
   \end{thm}
           
\section{$s(.,.)$-fractional Musielak-Sobolev spaces}\label{S3}
Due to the non-locality of the operator $\left( -\Delta\right)^{s(x,.)}_{a_{(x,.)}}$, we introduce  the  $s(.,.)$-fractional Musielak-Sobolev space as follows
 \begingroup\makeatletter\def\f@size{8}\check@mathfonts
$$W^{s(x,y)}{L_{\varPhi_{x,y}}}(\Omega)=\Bigg\{u\in L_{\widehat{\varPhi}_x}(\Omega) :  \int_{\Omega} \int_{\Omega} \varPhi_{x,y}\left( \dfrac{\lambda| u(x)- u(y)|}{|x-y|^{s(x,y)}}\right) \dfrac{dxdy}{|x-y|^N}< \infty \text{ for some } \lambda >0 \Bigg\}.
$$\endgroup

This space can be equipped with the norm
\begin{equation}\label{s2}
||u||_{s(x,y),\varPhi_{x,y}}=||u||_{\widehat{\varPhi}_x}+[u]_{s(x,y),\varPhi_{x,y}},
\end{equation}
where $[u]_{s(x,y),\varPhi_{x,y}}$ is the Gagliardo seminorm defined by 
$$[u]_{s(x,y),\varPhi_{x,y}}=\inf \Bigg\{\lambda >0 :  \int_{\Omega} \int_{\Omega} \varPhi_{x,y}\left( \dfrac{|u(x)- u(y)|}{\lambda|x-y|^{s(x,y)}}\right) \dfrac{dxdy}{|x-y|^N}\leqslant 1 \Bigg\}.
$$
To simplify notations, throughout the rest of this paper, we set
    $$D^{s(x,y)}u=\dfrac{u(x)-u(y)}{|x-y|^{s(x,y)}},~~ \text{ and }
     d\mu= \dfrac{dxdy}{|x-y|^N}. $$ 
\begin{rem}\text{ }\\
a$)-$ For the case: $\varPhi_{x,y}(t)=\varPhi(t)$, i.e. $\varPhi$ is independent of variables $x,y$, we can introduce the  $s(.,.)$-fractional Orlicz-Sobolev  spaces $W^{s(x,y)}L_\varPhi(\Omega)$ as follows
\begingroup\makeatletter\def\f@size{9}\check@mathfonts
$$
W^{s(x,y)}{L_\varPhi}(\Omega)
      =\Bigg\{u\in L_\varPhi(\Omega) :  \int_{\Omega} \int_{\Omega} \varPhi\left( \dfrac{\lambda| u(x)- u(y)|}{|x-y|^{s(x,y)}}\right)\dfrac{ dxdy}{|x-y|^N}< \infty \textnormal{ for some } \lambda >0 \Bigg\}.     
$$
       \endgroup
b$)-$ For the case: $\varPhi_{x,y}(t)=|t|^{p(x,y)}$ for all $(x,y)\in \overline{\Omega}\times\overline{\Omega}$,  where $p:\overline{\Omega}\times\overline{\Omega}\longrightarrow(1,+\infty)$ is a continuous bounded function such that
\begin{equation*} 
1<p^{-}=\underset{(x,y)\in \overline{\Omega}\times\overline{\Omega}}{\min}p(x,y)\leqslant p(x,y)\leqslant p^{+}=\underset{(x,y)\in \overline{\Omega}\times\overline{\Omega}}{\max}p(x,y)<+\infty,
\end{equation*}
and
\begin{equation*}
p ~\text{is symmetric, that is, }~~ p(x,y)=p(y,x),~~~ \text{for all }(x,y)\in\overline{\Omega}\times\overline{\Omega}.
\end{equation*}
If denoted by
$\bar{p}(x)=p(x,x)$ for all $x\in \overline{\Omega}.$ Then,
 we replace $L_{\varPhi_x}$ by $L^{\overline{p}(x)}$, and $W^{s(x,y)}L_{\varPhi_{x,y}}$ by $W^{s(x,y),p(x,y)}$ and we refer them as variable exponent Lebesgue spaces,  and $s(.,.)$-fractional Sobolev spaces with variable exponent respectively, $($see \cite{s1,s2.,s2}$)$ defined by
$$L^{\overline{p}(x)}(\Omega)=\bigg\{u:\Omega\longrightarrow \mathbb{R} ~~\text{measurable}: \int_{\Omega}|u(x)|^{\overline{p}(x)}dx<+\infty
\bigg\},$$ and 
 $$\hspace{-9cm}W=W^{s(x,y),p(x,y)}(\Omega)$$
\begingroup\makeatletter\def\f@size{9}\check@mathfonts $$\hspace*{0.6cm}=\bigg\{u\in L^{\bar{p}(x)}(\Omega): \int_{\Omega\times\Omega}\frac{|u(x)-u(y)|^{p(x,y)}}{\lambda^{p(x,y)}|x-y|^{s(x,y)p(x,y)+N}}~dxdy <+\infty~~ \text{for some}~~\lambda>0\bigg\}$$\endgroup
with the norm
$$\|u\|_{W}=\|u\|_{L^{\bar{p}(x)}(\Omega)}+[u]_{W},$$
where $ [.]_{W}$ is a Gagliardo seminorm with variable exponent given by $$[u]_{W}=[u]_{s(x,y),p(x,y)}= \inf \bigg\{\lambda>0:\int_{\Omega\times\Omega}\frac{|u(x)-u(y)|^{p(x,y)}}{\lambda^{p(x,y)}|x-y|^{N+s(x,y)p(x,y)}}~dxdy \leqslant1 \bigg\}.$$
\end{rem}
\begin{thm}
       Let $\Omega$ be an open subset of $\R^N$. The space $W^{s(x,y)}L_{\varPhi_{x,y}}(\Omega)$ is a Banach space with respect to the norm $(\ref{s2})$, and a  separable $($resp. reflexive$)$ space if and only if $\varPhi_{x,y} \in \Delta_2$ $($resp. $\varPhi_{x,y}\in \Delta_2 $ and $\overline{\varPhi}_{x,y}\in \Delta_2$$)$. Furthermore,
              if   $\varPhi_{x,y} \in \Delta_2$ and $\varPhi_{x,y}(\sqrt{t})$ is convex, then  the space $W^{s(x,y)}L_{\varPhi_{x,y}}(\Omega)$ is an uniformly convex space.\end{thm}
    Proof of this Theorem is similar to  \cite[Theorem 2.1]{benkirane}.          

\begin{thm}\label{s3.2}
        Let $\Omega$  be a bounded open
              subset of  $\R^N$. Then
          $$W^{s^+}L_{\varPhi_{x,y}}(\Omega) \hookrightarrow  W^{s(x,y)}L_{\varPhi_{x,y}}(\Omega) \hookrightarrow W^{s^-}L_{\varPhi_{x,y}}(\Omega).$$ \end{thm}

   \begin{proof}[\textbf{Proof}]
Let $u\in W^{s^+}L_{\varPhi_{x,y}}(\Omega)$ and $\lambda>0$, we have

  \begingroup\makeatletter\def\f@size{9}\check@mathfonts$$  
     \begin{aligned}
     \int_{\Omega} \int_{\Omega} \varPhi_{x,y}\left( \dfrac{|D^{s(x,y)}u|}{\lambda}\right)\dfrac{dxdy}{|x-y|^N}&=\int_{\Omega} \int_{\Omega} \varPhi_{x,y}\left( \dfrac{|D^{s^+}u|}{\lambda}\dfrac{1}{|x-y|^{s(x,y)-s^+}}\right)\dfrac{dxdy}{|x-y|^N}\\
     &\leqslant \int_{\Omega} \int_{\Omega} \varPhi_{x,y}\left( \dfrac{|D^{s^+}u|}{\lambda}\right)\dfrac{ dxdy}{|x-y|^{N+p(s(x,y)-s^+)}}\\
     &\leqslant \sup\limits_{\overline{\Omega}\times \overline{\Omega}}|x-y|^{p(s^+-s(x,y))} \int_{\Omega} \int_{\Omega} \varPhi_{x,y}\left( \dfrac{|D^{s^+}u|}{\lambda}\right)\dfrac{ dxdy}{|x-y|^{N}}\\
     \end{aligned}
     $$\endgroup
     where $p=\left\lbrace\varphi^-\text{ or } \varphi^+\right\rbrace$ is given by Lemma \ref{2.2..}.  This implies that 
$$[u]_{s(x,y),\varPhi_{x,y}}\leqslant \sup\limits_{\overline{\Omega}\times \overline{\Omega}}|x-y|^{p(s^+-s(x,y))}[u]_{s^+,\varPhi_{x,y}}.$$
So 
$$\|u\|_{s(x,y),\varPhi_{x,y}}\leqslant c\|u\|_{s^+,\varPhi_{x,y}},$$
where $c=\max\left\lbrace 1, \sup\limits_{\overline{\Omega}\times \overline{\Omega}}|x-y|^{p(s^+-s(x,y))}\right\rbrace$.

Now, Let $u\in W^{s(x,y)}L_{\varPhi_{x,y}}(\Omega)$ and $\lambda>0$, we have

  \begingroup\makeatletter\def\f@size{9}\check@mathfonts$$  
     \begin{aligned}
     \int_{\Omega} \int_{\Omega} \varPhi_{x,y}\left( \dfrac{|D^{s^-}u|}{\lambda}\right)\dfrac{dxdy}{|x-y|^N}&=\int_{\Omega} \int_{\Omega} \varPhi_{x,y}\left( \dfrac{|D^{s(x,y)}u|}{\lambda}\dfrac{1}{|x-y|^{s^--s(x,y)}}\right)\dfrac{dxdy}{|x-y|^N}\\
     &\leqslant \int_{\Omega} \int_{\Omega} \varPhi_{x,y}\left( \dfrac{|D^{s(x,y)}u|}{\lambda}\right)\dfrac{ dxdy}{|x-y|^{N+p(s^--s(x,y))}}\\
     &\leqslant \sup\limits_{\overline{\Omega}\times \overline{\Omega}}|x-y|^{p(s(x,y)-s^-)} \int_{\Omega} \int_{\Omega} \varPhi_{x,y}\left( \dfrac{|D^{s(x,y)}u|}{\lambda}\right)\dfrac{ dxdy}{|x-y|^{N}}.\\
     \end{aligned}
     $$\endgroup
  This implies that 
$$[u]_{s^-,\varPhi_{x,y}}\leqslant \sup\limits_{\overline{\Omega}\times \overline{\Omega}}|x-y|^{p(s(x,y)-s^-)}[u]_{s(x,y),\varPhi_{x,y}}.$$
So 
$$\|u\|_{s^-,\varPhi_{x,y}}\leqslant c\|u\|_{s(x,y),\varPhi_{x,y}},$$
where $c=\max\left\lbrace 1, \sup\limits_{\overline{\Omega}\times \overline{\Omega}}|x-y|^{p(s(x,y)-s^-)}\right\rbrace$.
\end{proof}
Now, combining Theorem \ref{s3.2} and Theorem \ref{th2.}, we obtain the following results.
\begin{coro}\label{s333}
 Let $\Omega$  be a bounded open
       subset of  $\R^N$ with $C^{0,1}$-regularity 
         and bounded boundary. If $(\ref{15})$ and  $(\ref{16n})$  hold, then 
      \begin{equation*}\label{s18}
       W^{s(x,y)}{L_{\varPhi_{x,y}}}(\Omega)\hookrightarrow L_ {\widehat{\varPhi}^*_{x,s^-}}(\Omega).
      \end{equation*}
               Also, the embedding
                \begin{equation*}\label{s27}
                 W^{s(x,y)}{L_{\varPhi_{x,y}}}(\Omega)\hookrightarrow L_{B_x}(\Omega),
                \end{equation*}
                is compact for all $B_x\prec\prec \widehat{\varPhi}^*_{x,s^-}$.
\end{coro}

 For any $u \in W^{s(x,y)}L_{\varPhi_{x,y}}(\Omega)$, we define the modular function on  $W^{s(x,y)}L_{\varPhi_{x,y}}(\Omega)$  as follows  
                    \begin{equation}\label{smodN}
        J(u)=\displaystyle\int_{\Omega} \int_{\Omega} \varPhi_{x,y}\left( \dfrac{ |u(x)- u(y)|}{|x-y|^{s(x,y)}}\right) \dfrac{dxdy}{|x-y|^N}+\int_{\Omega}\widehat{\varPhi}_{x}\left( |u(x)|\right) dx. \end{equation}  
                     
   An important role in manipulating  the $s(.,.)$-fractional Musielak-Sobolev spaces is played by the modular function $(\ref{smodN})$. It is worth noticing that the relation between the norm and the modular shows an equivalence between the topology defined by the norm and that defined by the modular.                                     
        \begin{prop}\label{smod}
         Assume that (\ref{v1}) is satisfied. Then for any $u \in W^{s(x,y)}L_{\varPhi_{x,y}}(\Omega)$, the following relations hold true:
           \begin{equation}\label{smod1}
     ||u||_{s(x,y),\varPhi_{x,y}}>1\Longrightarrow      ||u||_{s(x,y),\varPhi_{x,y}}^{\varphi^-} \leqslant  J(u)\leqslant  ||u||_{s(x,y),\varPhi_{x,y}}^{\varphi^+},
           \end{equation}
           \begin{equation}\label{smod2}
                ||u||_{s(x,y),\varPhi_{x,y}}<1\Longrightarrow    ||u||_{s(x,y),\varPhi_{x,y}}^{\varphi^+} \leqslant  J(u)\leqslant  ||u||_{s(x,y),\varPhi_{x,y}}^{\varphi^-}. \end{equation}
\end{prop}
 \begin{proof}[\textbf{Proof}] To simplify the notation, we take $\|u\|_{x,y}:=||u||_{s(x,y),\varPhi_{x,y}}$. First, we show that if $||u||_{x,y}>1$, then $J(u) \leqslant ||u||^{\varphi^+}$. Indeed, let $u\in W^{s(x,y)}L_{\varPhi_{x,y}}(\Omega)$ such that $||u||_{x,y}>1$. Using the definition of the Luxemburg norm and the relation $(\ref{r10})$, we get 
                                 \begingroup\makeatletter\def\f@size{9}\check@mathfonts  $$
                                     \begin{aligned}
              J(u)&=\displaystyle\int_{\Omega} \int_{\Omega} \varPhi_{x,y}\left(||u||_{x,y} \dfrac{ |u(x)- u(y)|}{||u||_{x,y}|x-y|^{s(x,y)}}\right) \dfrac{dxdy}{|x-y|^N}+\int_{\Omega}\widehat{\varPhi}_{x}\left(||u||_{x,y}\dfrac{ |u(x)|}{||u||_{x,y}}\right) dx\\
             &\leqslant ||u||_{x,y}^{\varphi^+}\displaystyle\int_{\Omega} \int_{\Omega} \varPhi_{x,y}\left( \dfrac{ |u(x)- u(y)|}{||u||_{x,y}|x-y|^{s(x,y)}}\right) \dfrac{dxdy}{|x-y|^N}+||u||_{x,y}^{\widehat{\varphi}^+}\int_{\Omega}\widehat{\varPhi}_{x}\left(\dfrac{ |u(x)|}{||u||_{x,y}}\right) dx\\
             &\leqslant ||u||_{x,y}^{\varphi^+}\left[ \displaystyle\int_{\Omega} \int_{\Omega} \varPhi_{x,y}\left( \dfrac{ |u(x)- u(y)|}{||u||_{x,y}|x-y|^{s(x,y)}}\right) \dfrac{dxdy}{|x-y|^N}+\int_{\Omega}\widehat{\varPhi}_{x}\left(\dfrac{ |u(x)|}{||u||_{x,y}}\right) dx\right] \\
                &\leqslant ||u||_{x,y}^{\varphi^+}.
                                     \end{aligned}
                                     $$\endgroup
        Next, assume that $||u||_{x,y}>1$. Let $\beta\in (1,||u||_{x,y})$, by  $(\ref{3.})$, we have
      \begingroup\makeatletter\def\f@size{9}\check@mathfonts   $$
      \begin{aligned}
      \displaystyle\int_{\Omega} \int_{\Omega}& \varPhi_{x,y}\left( \dfrac{ |u(x)- u(y)|}{|x-y|^{s(x,y)}}\right) \dfrac{dxdy}{|x-y|^N}+\int_{\Omega}\widehat{\varPhi}_{x}\left( |u(x)|\right) dx\\
      &\geqslant \beta^{\varphi^-}  \displaystyle\int_{\Omega} \int_{\Omega} \varPhi_{x,y}\left( \dfrac{ |u(x)- u(y)|}{\beta|x-y|^{s(x,y)}}\right) \dfrac{dxdy}{|x-y|^N}+\beta^{\widehat{\varphi}^-}\int_{\Omega}\widehat{\varPhi}_{x}\left( \dfrac{|u(x)|}{\beta}\right) dx\\
      & \geqslant \beta^{\varphi^-} \left( \displaystyle\int_{\Omega} \int_{\Omega} \varPhi_{x,y}\left( \dfrac{ |u(x)- u(y)|}{\beta|x-y|^{s(x,y)}}\right) \dfrac{dxdy}{|x-y|^N}+\int_{\Omega}\widehat{\varPhi}_{x}\left( \dfrac{|u(x)|}{\beta}\right) dx\right).
      \end{aligned}   $$\endgroup
      Since $\beta< ||u||_{x,y}$, we find 
   $$ \displaystyle\int_{\Omega} \int_{\Omega} \varPhi_{x,y}\left( \dfrac{ |u(x)- u(y)|}{\beta|x-y|^{s(x,y)}}\right) \dfrac{dxdy}{|x-y|^N}+\int_{\Omega}\widehat{\varPhi}_{x}\left( \dfrac{|u(x)|}{\beta}\right) dx>1.$$   
   Thus, we have
  $$  \displaystyle\int_{\Omega} \int_{\Omega} \varPhi_{x,y}\left( \dfrac{ |u(x)- u(y)|}{|x-y|^{s(x,y)}}\right) \dfrac{dxdy}{|x-y|^N}+\int_{\Omega}\widehat{\varPhi}_{x}\left( |u(x)|\right) dx \geqslant  \beta^{\varphi^-}.$$
  Letting $\beta \nearrow ||u||_{x,y}$, we deduce that $(\ref{smod1})$ holds  true.
  
    Next, we show that $J(u) \leqslant ||u||_{x,y}^{\varphi^-} \text{   for all  }  u\in W^{s(x,y)}L_{\varPhi_{x,y}}(\Omega) \text{ with }||u||_{x,y}<1$. 
                   Using the definition of the Luxemburg norm and $(\ref{r11})$, we obtain
                        \begingroup\makeatletter\def\f@size{10}\check@mathfonts $$
                      \begin{aligned}
                      J(u)&\leqslant ||u||_{x,y}^{\varphi^-}\displaystyle\int_{\Omega} \int_{\Omega} \varPhi_{x,y}\left( \dfrac{ |u(x)- u(y)|}{||u||_{x,y}|x-y|^{s(x,y)}}\right) \dfrac{dxdy}{|x-y|^N}+||u||_{x,y}^{\widehat{\varphi}^-}\int_{\Omega}\widehat{\varPhi}_{x}\left(\dfrac{ |u(x)|}{||u||_{x,y}}\right) dx\\
                      &\leqslant ||u||_{x,y}^{\varphi^-}\left[ \displaystyle\int_{\Omega} \int_{\Omega} \varPhi_{x,y}\left( \dfrac{ |u(x)- u(y)|}{||u||_{x,y}|x-y|^{s(x,y)}}\right) \dfrac{dxdy}{|x-y|^N}+\int_{\Omega}\widehat{\varPhi}_{x}\left(\dfrac{ |u(x)|}{||u||_{x,y}}\right) dx\right] \\
                      &\leqslant||u||_{x,y}^{\varphi^-}.
                      \end{aligned}
                      $$\endgroup

  Let $\xi\in (0,||u||_{x,y})$. From $(\ref{3.2})$, it follows that 
  {\small \begin{equation}\label{re1}
       \begin{aligned}
       \displaystyle\int_{\Omega} \int_{\Omega}& \varPhi_{x,y}\left( \dfrac{ |u(x)- u(y)|}{|x-y|^{s(x,y)}}\right) \dfrac{dxdy}{|x-y|^N}+\int_{\Omega}\widehat{\varPhi}_{x}\left( |u(x)|\right) dx\\
       & \geqslant \xi^{\varphi^+}  \displaystyle\int_{\Omega} \int_{\Omega} \varPhi_{x,y} \left( \dfrac{ |u(x)- u(y)|}{\xi|x-y|^{s(x,y)}}\right) \dfrac{dxdy}{|x-y|^N}+\xi^{\widehat{\varphi}^+} \int_{\Omega}\varPhi\left( \dfrac{|u(x)|}{\xi}\right) dx\\
        & \geqslant \xi^{\varphi^+}\left[   \displaystyle\int_{\Omega} \int_{\Omega} \varPhi_{x,y} \left( \dfrac{ |u(x)- u(y)|}{\xi|x-y|^{s(x,y)}}\right) \dfrac{dxdy}{|x-y|^N}+ \int_{\Omega}\varPhi\left( \dfrac{|u(x)|}{\xi}\right) dx\right] .
       \end{aligned}    \end{equation} }
       Defining     $v(x)=\dfrac{u(x)}{\xi}$ for all $x\in \Omega$. Then, $||v||_{x,y}=\dfrac{||u||_{x,y}}{\xi}>1$. Using relation $(\ref{mod1})$, we find 
        \begin{equation}\label{re2}
        \displaystyle\int_{\Omega} \int_{\Omega} \varPhi_{x,y}\left( \dfrac{ |v(x)- v(y)|}{|x-y|^{s(x,y)}}\right) \dfrac{dxdy}{|x-y|^N}+\int_{\Omega}\widehat{\varPhi}_{x}\left( |v(x)|\right) dx\geqslant ||v||_{x,y}^{\varphi^-}>1. \end{equation}
        Combining $(\ref{re1})$ and $(\ref{re2})$, we deduce that
        $$\displaystyle\int_{\Omega} \int_{\Omega} \varPhi_{x,y}\left( \dfrac{ |u(x)- u(y)|}{|x-y|^{s(x,y)}}\right) \dfrac{dxdy}{|x-y|^N}+\int_{\Omega}\widehat{\varPhi}_{x}\left( |u(x)|\right) dx\geqslant \xi^{\varphi^-}.$$
      Letting $\xi\nearrow ||u||_{x,y}$ in the above inequality, we obtain that relation $(\ref{smod2})$ holds true.   \end{proof} 
              
                 Similar to Proposition \ref{mod}, we obtain the following results.
                                                  
 \begin{prop} Assume that \hyperref[v1]{$(\varPhi_1)$} is satisfied, Then for any $u \in W^{s(x,y)}L_{\varPhi_{x,y}}(\Omega)$, the following assertions hold true:                     
                    \begin{equation*}\label{32}
                  [u]_{s(x,y),\varPhi_{x,y}}>1 \Longrightarrow  [u]^{\varphi^-}_{s(x,y),\varPhi_{x,y}}\leqslant \phi(u) \leqslant [u]^{\varphi^+}_{s(x,y),\varPhi_{x,y}},
                    \end{equation*} 
                    
                     \begin{equation*}\label{A33}
                     [u]_{s(x,y),\varPhi_{x,y}}<1 \Longrightarrow  [u]^{\varphi^+}_{s(x,y),\varPhi_{x,y}}\leqslant \phi(u) \leqslant [u]^{\varphi^-}_{s(x,y),\varPhi_{x,y}}, 
                       \end{equation*} 
         where $\phi (u)=  \displaystyle\int_{\Omega} \int_{\Omega} \varPhi_{x,y}\left( \dfrac{| u(x)- u(y)|}{|x-y|^{s(x,y)}}\right)\dfrac{dxdy}{|x-y|^N}$.       
                    \end{prop}

Now, we introduce a closed linear subspace of $W^{s(x,y)}L_{\varPhi_{x,y}}(\Omega)$ as follows
$$W_0^{s(x,y)}L_{\varPhi_{x,y}}(\Omega)=\left\lbrace u\in W^{s(x,y)}L_{\varPhi_{x,y}}(\R^N)~~ u=0~~\text{in} ~~ \R^N\backslash \Omega\right\rbrace.$$

Then we have the following generalized Poincar\'{e} type inequality.

\begin{thm}\label{spc}
                                 Let $\Omega$ be a bounded open subset of  $\R^N$ with $C^{0,1}$-regularity 
                                       and bounded boundary.
                                  Then there exists a positive constant $\gamma$ such that    
                        $$ ||u||_{\widehat{\varPhi}_x}\leqslant \gamma [u]_{s(x,y),\varPhi_{x,y}} \text{ for all   }  u \in  W^{s(x,y)}_0L_{\varPhi_{x,y}}(\Omega).$$
                                \end{thm}
   \begin{proof}[\textbf{Proof}]
   Let $u\in W_0^{s(x,y)}L_{\varPhi_{x,y}}(\Omega)$, by Theorem $\ref{s3.2}$, we have 
   \begin{equation}\label{s22}
   [u]_{s^-,\varPhi_{x,y}}\leqslant c[u]_{s(x,y),\varPhi_{x,y}},\end{equation}
   on the other hand, by Theorem \ref{pc}, there exists a positive constant $\gamma'$ such that  
        \begin{equation}\label{s33} ||u||_{\widehat{\varPhi}_x}\leqslant \gamma' [u]_{s^-,\varPhi_{x,y}} \text{ for all   }  u \in  W^{s^-}_0L_{\varPhi_{x,y}}(\Omega).\end{equation}
     Thus, we combining ($\ref{s22}$) with ($\ref{s33}$), we obtain 
      $$ ||u||_{\widehat{\varPhi}_x}\leqslant \gamma [u]_{s(x,y),\varPhi_{x,y}} \text{ for all   }  u \in  W^{s(x,y)}_0L_{\varPhi_{x,y}}(\Omega).$$
   with $\gamma=c\gamma'$. 
      \end{proof}                                
       Now, in order to study Problem \hyperref[P]{$(\mathcal{P}_{a})$}, it is important to encode the boundary condition $u=0$ in $\R^N\setminus \Omega$  in the weak formulation. In the scalar case, Servadei and Valdinoci \cite{servadei}
       introduced a new function spaces to study the variational functionals related to the fractional Laplacian by observing the interaction between $\Omega$ and $\R^N\setminus \Omega$. 
       Subsequently,  inspired by the work of Servadei and Valdinoci \cite{servadei}, Azroul et al in  \cite{SRH},  have
      introduced the  fractional Sobolev space with variable exponent,  to study the variational functionals related to the fractional $p(x,.)$-Laplacian operator by observing the interaction between $\Omega$ and $\R^N\setminus \Omega$. Motivated by the above papers, and due to the nonlocality of the  $s(.,.)$-order operator $(-\Delta)^{s(x,.)}_{a_(x,.)}$, we introduce the following $s(.,.)$-fractional Orlicz-Sobolev space 
       as follows
        \begingroup\makeatletter\def\f@size{8}\check@mathfonts $$W^{(x,y)}L_{\varPhi_{w,y}}(Q)=\Bigg\{u\in L_{\varPhi_{x,y}}(\Omega) ~ :~ \int_{Q}  \varPhi_{x,y}\left( \dfrac{\lambda|u(x)- u(y)|}{|x-y|^{s(x,y)}}\right) \dfrac{dxdy}{|x-y|^N}< \infty~~ \text{for some }\lambda>0 \Bigg\},
   $$\endgroup
                                                             where $Q=\R^{2N}\setminus (C\Omega\times C\Omega)$ with $C\Omega=\R^N \setminus \Omega$. This spaces are equipped with the norm,
                                                              \begin{equation}\label{an6}
                                                              ||u||=||u||_{\widehat{\varPhi}_x}+[u],
                                                              \end{equation}
                                                              where $[.]$ is the Gagliardo seminorm, defined by 
                                                              $$[u]=\inf \Bigg\{\lambda > 0 :  \int_{Q} \varPhi_{x,y}\left( \dfrac{|u(x)- u(y)|}{\lambda|x-y|^{s(x,y)}}\right) \dfrac{dxdy}{|x-y|^N}\leqslant 1 \Bigg\}.
                                                              $$ 
                                                               Similar to the spaces $(W^{s(x,y)}L_{\varPhi_{x,y}}(\Omega), \|.\|_{s(x,y),\varPhi_{x,y}})$ we have that $(W^{s(x,y)}L_{\varPhi_{x,y}}(Q),  \|.\|)$ is a separable reflexive Banach spaces.
                                                               
                                                               Now, let $W_0^{s(x,y)}L_{\varPhi_{x,y}}(Q)$ denotes the following linear subspace of $W^{s(x,y)}L_{\varPhi_{x,y}}(Q),$
                                                               $$W_0^{s(x,y)}L_{\varPhi_{x,y}}(Q)=\left\lbrace u\in W^{s(x,y)}L_{\varPhi_{x,y}}(Q) ~:~ u=0 \text{ a.e in } \R^N \setminus \Omega\right\rbrace $$
                                                               with the norm
                                                               $$[u]=\inf \Bigg\{\lambda > 0 :  \int_{Q} \varPhi_{x,y}\left( \dfrac{|u(x)- u(y)|}{\lambda|x-y|^{s(x,y)}}\right) \dfrac{dxdy}{|x-y|^N}\leqslant 1 \Bigg\}.
                              $$

                                                               In the following theorem, we compare the spaces $W^{s(x,y)}L_{\varPhi_{x,y}}(\Omega)$ and $W^{s(x,y)}L_{\varPhi_{x,y}}(Q)$.
                                                               \begin{thm}\label{san2} The following assertions hold:
                                                               \begin{itemize}
                                                               \item[1)] The continuous embedding $$W^{s(x,y)}L_{\varPhi_{x,y}}(Q)\subset W^{s(x,y)}L_{\varPhi_{x,y}}(\Omega)$$
                                                                holds true.\\
                                                               \item[2)] If $u\in W_0^{s(x,y)}L_{\varPhi_{x,y}}(Q)$, then $u\in W^{s(x,y)}L_{\varPhi_{x,y}}(\R^N)$ and 
                                                                $$||u||_{s(x,y),\varPhi_{x,y}}\leqslant ||u||_{W^{s(x,y)}L_{\varPhi_{x,y}}(\R^N)}=||u||.$$
                                                                 \end{itemize}
                                                                 \end{thm}
                                                                 \begin{proof}[\textbf{Proof}]
                                                                 $1)$ Let $u\in W^{s(x,y)}L_{\varPhi_{x,y}}(Q)$, since $\Omega\times \Omega\subsetneq Q,$ then for all $\lambda>0$ we have 
                                                               \begingroup\makeatletter\def\f@size{10}\check@mathfonts    \begin{equation}\label{an1}
         \int_{\Omega}\int_{\Omega} \varPhi_{x,y}\left( \dfrac{|u(x)- u(y)|}{\lambda|x-y|^{s(x,y)}}\right) \dfrac{dxdy}{|x-y|^N}\leqslant \int_{Q} \varPhi_{x,y}\left( \dfrac{|u(x)- u(y)|}{\lambda|x-y|^{s(x,y)}}\right) \dfrac{dxdy}{|x-y|^N}.    \end{equation}\endgroup
                                                                  We set 
                                                                  $$\mathcal{A}^{s(x,y)}_{\lambda,\Omega\times \Omega}=\Bigg\{\lambda > 0 :  \int_{\Omega}\int_{\Omega} \varPhi_{x,y}\left( \dfrac{|u(x)- u(y)|}{\lambda|x-y|^{s(x,y)}}\right) \dfrac{dxdy}{|x-y|^N}\leqslant 1 \Bigg\}$$
                                                                  and 
                                                                   $$\mathcal{A}^{s(x,y)}_{\lambda,Q}=\Bigg\{\lambda > 0 :  \int_{Q} \varPhi_{x,y}\left( \dfrac{|u(x)- u(y)|}{\lambda|x-y|^{s(x,y)}}\right) \dfrac{dxdy}{|x-y|^N}\leqslant 1 \Bigg\}.$$
                                                                   By $(\ref{an1})$, it is easy to see that $\mathcal{A}^{s(x,y)}_{\lambda,Q}\subset \mathcal{A}^{s(x,y)}_{\lambda,\Omega\times\Omega}$. Hence 
                                                                   \begin{equation}\label{an}
                                                                  [u]_{s(x,y),\varPhi_{x,y}}=\inf\limits_{\lambda>0}\mathcal{A}^{s(x,y)}_{\lambda,\Omega\times\Omega}\leqslant[u]=\inf\limits_{\lambda>0}\mathcal{A}^{s(x,y)}_{\lambda,Q}. \end{equation}
                                                                Consequently, by definitions of the norms $\|u\|_{s(x,y),\varPhi_{x,y}}$ and $\|u\|,$ we obtain
                                                                  $$ \|u\|_{s(x,y),\varPhi_{x,y}}\leqslant \|u\|<\infty.$$
                                                                 $2)$ Let $u\in W_0^{s(x,y)}L_{\varPhi_{x,y}}(Q)$, then $u=0$ in $\R^N\setminus \Omega$. So, $\|u\|_{L_{\widehat{\varPhi}_x}(\Omega)}=\|u\|_{L_{\widehat{\varPhi}_x}(\R^N)}.$ Since 
                                                                  $$\int_{\R^{2N}} \varPhi_{x,y}\left( \dfrac{|u(x)- u(y)|}{\lambda|x-y|^{s(x,y)}}\right) \dfrac{dxdy}{|x-y|^N}=\int_{Q} \varPhi_{x,y}\left( \dfrac{|u(x)- u(y)|}{\lambda|x-y|^{s(x,y)}}\right) \dfrac{dxdy}{|x-y|^N}$$
        for all $\lambda>0$. Then $[u]_{W^{s(x,y)}L_{\varPhi_{x,y}}(\R^N)}=[u]$. Thus, we get
                                                                  
                                                                   $$||u||_{s(x,y),\varPhi_{x,y}}\leqslant ||u||_{W^{s(x,y)}L_{\varPhi_{x,y}}(\R^N)}=||u||.$$
                                                                  \end{proof}
                                                                  \begin{coro}\label{ann}(Poincar\'{e} inequality)
                                                                  Let $\Omega$ be a bounded open subset of  $\R^N$ with $C^{0,1}$-regularity 
                                                                                                        and bounded boundary. Then there exists a positive
                                                                  constant $c$ such that,
                                                                  $$
                                                                  \|u\|_{\widehat{\varPhi}_x}\leqslant c[u] ~~~~\forall u\in W^{s(x,y)}_0L_{\varPhi_{x,y}}(Q).$$
                                                                  \end{coro}
                                                                   \begin{proof}[\textbf{Proof}]
                                                               Let $u\in W^{s(x,y)}_0L_{\varPhi_{x,y}}(Q)$, by Theorem $\ref{san2}$, we have $u\in W^{s(x,y)}_0L_{\varPhi_{x,y}}(\Omega)$. Then by Theorem \ref{spc},  there exists a positive
                                                                   constant $\gamma$ such that,
                                                                   $$
                                                                   \|u\|_{\widehat{\varPhi}_x}\leqslant \gamma[u]_{s(x,y),\varPhi_{x,y}}.$$
                                                                  Combining the above inequality with $(\ref{an})$, we obtain that 
                                                                   $$
                                 \|u\|_{\widehat{\varPhi}_x}\leqslant c[u] ~~~~\forall u\in W^{s(x,y)}_0L_{\varPhi_{x,y}}(Q).$$
                                                                   \end{proof} 
\begin{rem}
From Corollary \ref{ann}, we deduce that $[.]$ is a norm on $W^{s(x,y)}_0L_{\varPhi_{x,y}}(Q)$ which is equivalent to the
 norm $\|.\|$.   
\end{rem}  
  \section{Existence results and proofs}\label{S4}
  In this section,  we analyze problem \hyperref[P]{$(\mathcal{P}_{a})$}. under the following basic assumptions
   \begin{equation}\label{7}
   q^-<\varphi^-
   \end{equation}
   and 
   \begin{equation}\label{8}
   \lim_{t\rightarrow \infty}\left(\sup\limits_{x\in \overline{\Omega}}\dfrac{|t|^{q^+}}{(\widehat{\varPhi}_{x,s^-})_*(kt)}\right)=0 ~~\forall k>0.
   \end{equation}
 The dual space of $\left(W^{s(x,y)}_0L_{\varPhi_{x,y}}(Q),||.||\right) $  is denoted by $\left(\left( W^{s(x,y)}_0L_{\varPhi_{x,y}}(Q)\right) ^*,||.||_{*}\right) $.

  \begin{defini}
  We say that $\lambda\in \R$ is an eigenvalue of Problem \hyperref[P]{$(\mathcal{P}_{a})$} if there exists $u\in W_0^{s(x,y)}L_{\varPhi_{x,y}}(Q)\setminus \left\lbrace 0\right\rbrace$ such that 
  $$\int_{Q}a_{x,y}(|D^{s(x,y)}u|)  D^{s(x,y)}u D^{s(x,y)}vd\mu-\lambda\int_{\Omega}|u|^{q(x)-2}uvdx=0$$
  for all $v\in W_0^{s(x,y)}L_{\varPhi_{x,y}}(Q)$.
  \end{defini}
  
   We point that if $\lambda$ is an eigenvalue of Problem \hyperref[P]{$(\mathcal{P}_{a})$} then the corresponding $u\in W_0^{s(x,y)}L_{\varPhi_{x,y}}(Q)\setminus\left\lbrace 0\right\rbrace $ is a weak solution of \hyperref[P]{$(\mathcal{P}_{a})$}.\\
   
    Our main results is given by the following theorem.
     \begin{thm}\label{mn}
      There exists $\lambda_*>0$ such that for any $\lambda\in (0,\lambda_*)$ is an eigenvalue of Problem \hyperref[P]{$(\mathcal{P}_{a})$}.
         
     \end{thm}
     \begin{rem}\label{rem}
     By $(\ref{8})$, we can apply Theorem $\ref{san2}$ and Corollary $\ref{s333}$ we obtain that $W_0^{s(x,y)}L_{\varPhi_{x,y}}(Q)$ is compactly embedded in $L^{q+}(\Omega)$. That fact combined with the continuous embedding of $L^{q^+}(\Omega)$ in $L^{q(x)}(\Omega)$, ensures that $W^{s(x,y)}_0L_{\varPhi_{x,y}}(Q)$ is compactly embedded in $L^{q(x)}(\Omega)$.
     \end{rem}
   Next for all $\lambda\in \R$, we define the energetic function associated with problem \hyperref[P]{$(\mathcal{P}_{a})$} $J_\lambda : W^{s(x,y)}_0L_{\varPhi_{x,y}}(Q)\rightarrow \R$, as
   $$J_\lambda(u)=\int_{Q}\varPhi_{x,y}\left( \dfrac{|u(x)-u(y)|}{|x-y|^{s(x,y)}}\right) d\mu-\lambda\displaystyle\int_{\Omega}\dfrac{1}{q(x)}|u|^{q(x)}dx.$$
   
  By a standard argument to  \cite{sr5} and \cite{3}, we have $J_\lambda\in C^1(W^{s(x,y)}_0L_{\varPhi_{x,y}}(Q),\R)$,
  $$\left\langle J_\lambda'(u),v\right\rangle =\int_{Q} a_{x,y}(|D^{s(x,y)}u|)  D^{s(x,y)}u D^{s(x,y)}vd\mu-\lambda\displaystyle\int_{\Omega}|u|^{q(x)-2}uvdx.$$
    
  \begin{lem}\label{anlem3}
 Assume that the hypothesis of Theorem $\ref{mn}$ is fulfilled. Then, there exists $\lambda_*>0$ such that for any
 $\lambda\in (0,\lambda_*)$, there are $\rho, \alpha>0$, such that $J_\lambda(u)\geqslant \alpha>0$ for any $u\in W^{s(x,y)}_0L_{\varPhi_{x,y}}(Q)$ with $||u||=\rho$.
   \end{lem}  
   \begin{proof}[\textbf{Proof}]
   Since $W^{s(x,y)}_0L_{\varPhi_{x,y}}(Q)$ is continuously embedded in $L^{q(x)}(\Omega)$, it follows that there exists a positive constant $c_1$  such that
    \begin{equation}\label{an11}
   ||u||\geqslant c_1||u||_{q(x)} ~~\forall u\in W^{s(x,y)}_0L_{\varPhi_{x,y}}(Q)\end{equation} 
    we fix $\rho \in (0,1)$ such that $\rho<\dfrac{1}{c_1}$. Then relation $(\ref{an11})$ implies that 
    $$\|u\|_{q(x)}<1~~\text{for all } u\in W^{s(x,y)}_0L_{\varPhi_{x,y}}(Q) ~~\text{with } ||u||=\rho.$$
    Then, we can apply Proposition $\ref{2.2}$, and we have
    \begin{equation}\label{an12}
    \int_\Omega |u(x)|^{q(x)}dx\leqslant \|u\|_{q(x)}^{q^-}~~\text{ for all } u\in W^{s(x,y)}_0L_{\varPhi_{x,y}}(Q)~~ \text{ with } ||u||=\rho.
    \end{equation}
    Relation $(\ref{an11})$  and $(\ref{an12})$ implies that
    {\small \begin{equation}\label{an13}
       \int_\Omega |u(x)|^{q(x)}dx\leqslant c_1^{q^-}\|u\|^{q^-}~~\text{ for all } u\in W^{s(x,y)}_0L_{\varPhi_{x,y}}(Q)~~ \text{ with } ||u||=\rho.
       \end{equation}}
    Taking into account Relation $(\ref{an13})$, we deduce that for any $u \in W^{s(x,y)}_0L_{\varPhi_{x,y}}(Q)$ with  $||u||=\rho$, the following inequalities hold true:
    $$
    \begin{aligned}
    J_\lambda(u) &\geqslant  \|u\|^{\varphi^+}-\dfrac{\lambda}{q^-}\int_\Omega|u(x)|^{q(x)}dx\\
    & \geqslant \|u\|^{\varphi^+}-\dfrac{\lambda c_1^{q-}}{q^-}\|u\|^{q^-}\\
   & =\rho^{q^-}\left(\rho^{\varphi^+-q^-}- \dfrac{\lambda c_1^{q-}}{q^-}\right).
    \end{aligned}
    $$
    Hence, if we define
   \begin{equation}\label{an40}
   \lambda_*=\dfrac{\rho^{\varphi^+-q^-}}{2c_1^{q^-}}q^-.\end{equation}
   Then for any $\lambda\in (0,\lambda_*)$ and $u\in W^{s(x,y)}_0L_{\varPhi_{x,y}}(Q)$ with  $||u||=\rho$, we have
   $$J_\lambda(u)\geqslant \alpha>0,$$
   such that 
   $$\alpha=\dfrac{\rho^{\varphi^+}}{2}.$$
   This completes the proof.
      \end{proof}  
      
    \begin{lem}\label{anlem4}
      Assume that the hypothesis of Theorem $\ref{mn}$ is fulfilled. Then, there exists $\phi>0$ such that $\phi\geqslant 0$,  $\phi\neq 0$, and $J_\lambda(t\phi)<0$ for $t>0$ small enough.
        \end{lem}    
      \begin{proof}[\textbf{Proof}] 
    By assumption   $(\ref{7})$ we can chose $\varepsilon_0>0$ such that $q^-+\varepsilon_0<\varphi^-$. On the other hand, since $q\in C(\overline{\Omega})$, it follows that there exists an open set $\Omega_0\subset \Omega$ such that $|q(x)-q^-|<\varepsilon_0$ for all $x\in \Omega_0$. Thus, $q(x)\leqslant q^-+\varepsilon_0<\varphi^-$ for all $x\in \Omega_0$. 
    Let $\phi\in C_0^\infty(\Omega)$ be such that $supp(\phi)\supset \overline{\Omega_0}$, $\phi(x)=1$ for all $x\in \overline{\Omega_0}$, and $0\leqslant \phi\leqslant 1$ in $\overline{\Omega_0}$. Then  for any $t\in(0,1)$, we have
    $$
    \begin{aligned}
    J_\lambda(t\phi)& =\displaystyle\int_{Q}\varPhi_{x,y}\left(t|D^{s(x,y)}\phi| \right)d\mu-\lambda\displaystyle\int_{\Omega}\dfrac{1}{q(x)}t^{q(x)}|\phi|^{q(x)}dx\\
    &\leqslant \displaystyle\int_{Q}t^{\varphi^-}\varPhi_{x,y}\left(|D^{s(x,y)}\phi| \right)d\mu-\lambda\displaystyle\int_{\Omega_0}\dfrac{t^{q(x)}}{q(x)}|\phi|^{q(x)}dx\\
        &\leqslant t^{\varphi^-}\displaystyle\int_{Q}\varPhi_{x,y}\left(|D^{s(x,y)}\phi| \right)d\mu-\dfrac{\lambda t^{q^-+\varepsilon_0}}{q^+}\displaystyle\int_{\Omega_0}|\phi|^{q(x)}dx.
        \end{aligned}
       $$
       Therefore $J_\lambda(t\phi)<0$ for $t<\delta^{1/(\varphi^--q^--\varepsilon_0)}$ with
       $$0<\delta<\min\left\lbrace 1,~~ \dfrac{\dfrac{\lambda}{q^+}\displaystyle\int_{\Omega_0}|\phi|^{q(x)}dx}{\displaystyle\int_{Q}\varPhi_{x,y}\left(|D^{s(x,y)}\phi| \right)d\mu}\right\rbrace. $$
      This is possible, since we claim that 
      $$  \displaystyle\int_{Q}\varPhi_{x,y}\left(|D^{s(x,y)}\phi| \right)d\mu>0.$$
      Indeed, it is clear that
      $$\int_{\Omega_0}|\phi|^{q(x)}dx\leqslant \int_{\Omega}|\phi|^{q(x)}dx\leqslant \int_{\Omega}|\phi|^{q^-}dx.$$
      On the other hand, since  $W^{s(x,y)}_0L_{\varPhi_{x,y}}(Q)$ is continuously embedded in $L^{q^-}(\Omega)$, it follows that there exists a positive constant $c$  such that 
      $$\|\phi\|_{q^-}\leqslant c ||\phi||.$$
      The last two inequalities imply that
      $$\|\phi\|>0$$
      and combining this fact with Proposition $\ref{smod}$, the claim follows at once. The proof of the lemma is now completed.
      \end{proof}  
   \begin{proof}[\textbf{Proof of Theorem $\ref{mn}$}]    
      Let $\lambda_*>0$ be defined as in $(\ref{an40})$ and $\lambda\in (0,\lambda_*)$. By Lemma $\ref{anlem3}$ it follows that on the boundary oh the ball centered in the origin and of radius $\rho$ in $W^{s(x,y)}_0L_{\varPhi_{x,y}}(Q)$, denoted by $B_\rho(0)$, we have 
            $$\inf\limits_{\partial B_\rho(0)}J_\lambda>0.$$
         On the other hand, by Lemma $\ref{anlem4}$, there exists $\phi \in W^{s(x,y)}_0L_{\varPhi_{x,y}}(Q)$ such that $J_\lambda(t\phi)<0$ for all $t>0$ small enough. Moreover for any $u\in B_\rho(0)$, we have 
         $$
                  \begin{aligned}
              J_\lambda(u)\geqslant  \|u\|^{\varphi^-}-\dfrac{\lambda c_1^{q^-}}{q^-}\|u\|^{q^-}.
                    \end{aligned}
                               $$
         It follows that
         $$-\infty<c:=\inf\limits_{\overline{B_\rho(0)}} J_\lambda<0.$$   
         We let now $0<\varepsilon <\inf\limits_{\partial  B_\rho(0)}  J_\lambda -  \inf\limits_{B_\rho(0)} J_\lambda.$    Applying Theorem $\ref{th1}$ to the functional 
         $J_\lambda : \overline{B_\rho(0)}\longrightarrow \R$, we find $u_\varepsilon \in \overline{B_\rho(0)}$ such that 
          $$
               \left\{ 
                    \begin{array}{clclc}
                  J_\lambda(u_\varepsilon)&<\inf\limits_{\overline{B_\rho(0)}} J_\lambda+\varepsilon,& \\\\
                    J_\lambda(u_\varepsilon)&< J_\lambda(u)+\varepsilon ||u-u_\varepsilon||,& \text{  } u\neq u_\varepsilon.
                    \end{array}
                    \right. 
                 $$
          Since  $J_\lambda(u_\varepsilon)\leqslant  \inf\limits_{\overline{B_\rho(0)}} J_\lambda+\varepsilon\leqslant \inf\limits_{B_\rho(0)} J_\lambda+\varepsilon < \inf\limits_{\partial  B_\rho(0)}  J_\lambda$, we deduce $u_\varepsilon  \in B_\rho(0)$. 
          
          Now, we define $\Lambda_\lambda :  \overline{B_\rho(0)}\longrightarrow \R$ by 
          $$\Lambda_\lambda(u)=J_\lambda(u)+\varepsilon||u-u_\varepsilon||.$$
          It's clear that $u_\varepsilon$ is a minimum point of $\Lambda_\lambda$ and then
          $$\dfrac{\Lambda_\lambda(u_\varepsilon+t v)-\Lambda_\lambda(u_\varepsilon)}{t}\geqslant 0$$ 
          for small $t>0$, and any $v\in B_\rho(0).$ The above relation yields 
              $$\dfrac{J_\lambda(u_\varepsilon+t v)-J_\lambda(u_\varepsilon)}{t}+\varepsilon||v||\geqslant 0.$$
              Letting $t\rightarrow$ it follows that $\left\langle J'_{\lambda}(u_\varepsilon),v\right\rangle +\varepsilon ||v||>0$ and we infer that $$||J'_{\lambda}(u_\varepsilon)||_{*}\leqslant \varepsilon.$$
               We deduce that there exists a sequence $\left\lbrace u_n\right\rbrace \subset B_\rho(0)$ such that 
               \begin{equation}\label{an10}
               J_\lambda(u_n) \longrightarrow c \text{ and } J'_\lambda(u_n)\longrightarrow 0.
               \end{equation}
            It is clear that $\left\lbrace u_n\right\rbrace $ is bounded in $W^{s(x,y)}_0L_{\varPhi_{x,y}}(Q)$. Thus, there exists $u_0\in W^{s(x,y)}_0L_{\varPhi_{x,y}}(Q)$, such that up to a subsequence   $\left\lbrace u_n\right\rbrace $ converges weakly to $u_0$ in $W^{s(x,y)}_0L_{\varPhi_{x,y}}(Q)$.

             On the other hand, since $W^{s(x,y)}_0L_{\varPhi_{x,y}}(Q)$ is compactly embedded in $L^{q(x)}(\Omega)$, it follows that $\left\lbrace u_n\right\rbrace $ converges strongly to $u_0$ in $L^{q(x)}(\Omega)$.
             Then by H\"{o}lder inequality, we have that 
             $$ \lim\limits_{n\rightarrow \infty}\int_\Omega |u_n|^{q(x)-2}u_n(u_n-u_0)dx=0.$$
             This fact and relation $(\ref{an10})$, implies that $$\lim\limits_{n\rightarrow \infty}\left\langle J'_\lambda(u_n), u_n-u_0\right\rangle =0.$$
              Thus we deduce that 
              \begin{equation}\label{an31}
              \lim\limits_{n\rightarrow \infty} \int_{Q}a_{x,y}(|D^{s(x,y)}u_n|)D^{s(x,y)}u_n\left( D^{s(x,y)}u_n-D^{s(x,y)}u_0 \right)d\mu =0.
              \end{equation}
              Since $\left\lbrace u_n\right\rbrace $ converge weakly to $u_0$ in $W^{s(x,y)}_0L_{\varPhi_{x,y}}(Q)$, by relation $(\ref{an31})$, we find that 
              \begingroup\makeatletter\def\f@size{8}\check@mathfonts \begin{equation}\label{an32}
               \lim\limits_{n\rightarrow \infty} \int_{Q}\left( a_{x,y}(|D^{s(x,y)}u_n|)D^{s(x,y)}u_n-a_{x,y}(|D^{s(x,y)}u_0|)D^{s(x,y)}u_0\right) \left( D^{s(x,y)}u_n-D^{s(x,y)}u_0 \right)d\mu =0.
              \end{equation}\endgroup
              Since, $\varPhi_{x,y}$ is convex, we have
              \begingroup\makeatletter\def\f@size{9}\check@mathfonts $$\varPhi_{x,y}(|D^{s(x,y)}u|)\leqslant \varPhi_{x,y}\left( \dfrac{|D^{s(x,y)}u+D^{s(x,y)}v|}{2}\right) +a_{x,y}(|D^{s(x,y)}u|)D^{s(x,y)}u\dfrac{D^{s(x,y)}u-D^{s(x,y)}v}{2}$$\endgroup
                \begingroup\makeatletter\def\f@size{9}\check@mathfonts$$\varPhi_{x,y}(|D^{s(x,y)}v|)\leqslant \varPhi_{x,y}\left( \dfrac{|D^{s(x,y)}u+D^{s(x,y)}v|}{2}\right) +a_{x,y}(|D^{s(x,y)}v|)D^{s(x,y)}v\dfrac{D^{s(x,y)}v-D^{s(x,y)}u}{2}$$\endgroup
    for every $u,v \in W^{s(x,y)}_0L_{\varPhi_{x,y}}(Q)$.
     Adding the above two relations and integrating over $Q$, we find that
 \begingroup\makeatletter\def\f@size{7}\check@mathfonts \begin{equation}\label{an33}
     \begin{aligned}
     \dfrac{1}{2}&\int_{Q} \left( a_{x,y}(|D^{s(x,y)}u|)D^{s(x,y)}u-a_{x,y}(|D^{s(x,y)}v|)D^{s(x,y)}v\right) \left(D^{s(x,y)}u-D^{s(x,y)}v \right)d\mu\\
     &\geqslant \int_Q\varPhi_{x,y}(|D^{s(x,y)}u|) d\mu+ \int_Q\varPhi_{x,y}(|D^{s(x,y)}v|) d\mu- 2\int_Q\varPhi_{x,y}\left( \dfrac{|D^{s(x,y)}u-D^{s(x,y)}v|}{2}\right)d\mu,
       \end{aligned}
   \end{equation}\endgroup      
for every $u,v \in W^{s(x,y)}_0L_{\varPhi_{x,y}}(Q)$.
On the other hand, since for each, we know that $\varPhi_{x,y}~:~ [0,\infty) \rightarrow \R$ is an increasing continuous function, with $\varPhi_{x,y}(0)=0$. Then by the conditions \hyperref[v1]{$(\varPhi_1)$} and \hyperref[f2.]{$(\varPhi_2)$}, we can apply \cite[Lemma 2.1]{Lam}  in order to obtain 

 \begingroup\makeatletter\def\f@size{9}\check@mathfonts \begin{equation}\label{an34}
 \begin{aligned}
 \dfrac{1}{2}&\left[ \int_Q\varPhi_{x,y}(|D^{s(x,y)}u|) d\mu+ \int_Q\varPhi_{x,y}(|D^{s(x,y)}v|) d\mu\right] \\
& \geqslant \int_Q\varPhi_{x,y}\left( \dfrac{|D^{s(x,y)}u+D^{s(x,y)}v|}{2}\right)d\mu+\int_Q\varPhi_{x,y}\left( \dfrac{|D^{s(x,y)}u-D^{s(x,y)}v|}{2}\right)d\mu
 \end{aligned}
 \end{equation}\endgroup
 for every $u,v \in W^{s(x,y)}_0L_{\varPhi_{x,y}}(Q)$.
 By $(\ref{an33})$ and $(\ref{an34})$,  we have
  \begingroup\makeatletter\def\f@size{9}\check@mathfonts\begin{equation}\label{an35}
 \begin{aligned}
 & \int_{Q} \left( a_{x,y}(|D^{s(x,y)}u|)D^{s(x,y)}u-a_{x,y}(|D^{s(x,y)}v|)D^{s(x,y)}v\right) \left(D^{s(x,y)}u-D^{s(x,y)}v \right)d\mu\\
 &\geqslant 4\int_Q\varPhi_i\left( \dfrac{|D^{s(x,y)}u-D^{s(x,y)}v|}{2}\right)d\mu
 \end{aligned}
 \end{equation}\endgroup
 for every $u,v \in W^{s(x,y)}_0L_{\varPhi_{x,y}}(Q)$.
 
 Relations $(\ref{an32})$ and $(\ref{an35})$ show that $\left\lbrace u_n\right\rbrace $ converge strongly to $u_0$ in $W^{s(x,y)}_0L_{\varPhi_{x,y}}(Q)$. Then by relation $(\ref{an10})$, we have 
$$J_\lambda(u_0)=c_1>0 ~~\text{and}~~ J'_\lambda(u_0)=0.$$
   Then, $u_0$ is a nontrivial weak solution for Problem \hyperref[P]{$(\mathcal{P}_{a})$}. This complete the proof.
                \end{proof}
       
           \section{Examples}  \label{S5} 
           
           In this section we point certain examples of functions $\varphi_{x,y}$ and $\varPhi_{x,y}$  which illustrate the results of this paper.
         
    \begin{ex}
    		As a first example, we can take

           $$\varphi_{x,y}(t)=\varphi_{1}(x,y,t) =p(x,y)\dfrac{|t|^{p(x,y)-2}t}{\log (1+|t|)}  ~~\text{ for all}~~ t\geqslant 0,$$ 
           and thus, 
        $$\varPhi_{x,y}(t)=p(x,y)\dfrac{|t|^{p(x,y)}}{\log (1+|t|)}+\int_{0}^{|t|}\dfrac{\tau^{p(x,y)}}{(1+\tau)(\log(1+\tau))^2}d\tau,$$     
        with $p\in C(Q)$ satisfies $2\leqslant p(x,y) <N$ for all $(x,y)\in  Q$.
    \end{ex}
        Then, in this case problem  \hyperref[P]{$(\mathcal{P}_a)$} becomes 
        $$\label{P2}
         (\mathcal{P}_{1}) \hspace*{0.5cm} \left\{ 
           \begin{array}{clclc}
        
        \left(-\Delta\right) _{\varphi_{1}}^{s(x,.)} u & = &   \lambda |u|^{q(x)-2}u    & \text{ in }& \Omega \\\\
           u & = & 0 \hspace*{0.2cm} \hspace*{0.2cm} & \text{ in } & \R^N\setminus \Omega,
           \end{array}
           \right. 
        $$ 
       with 
        $$\left(-\Delta\right) _{\varphi_{1}}^{s(x,.)} u(x)=p.v.\int_{\Omega}\dfrac{p(x,y)|D^{s(x,y)}u|^{p(x,y)-2}D^{s(x,y)}u}{\log(1+|D^{s(x,y)}u|)|x-y|^{N+s(x,y)}}~dy~~ \text{ for all  }  x \in \Omega.$$  
       It easy to see that $\varPhi_{x,y}$ is a Musielak function and satisfy condition \hyperref[v3]{$(\varPhi_3)$}.
        Moreover for each $(x,y)\in  Q$ fixed, by Example 3 on p 243 in \cite{cl}, we have 
        $$p(x,y)-1\leqslant \dfrac{t\varphi_{x,y}(t)}{\varPhi_{x,y}(t)}\leqslant p(x,y) ~~\forall (x,y) \in Q~~\forall t\geqslant 0.$$
        Thus, \hyperref[v1]{$(\varPhi_1)$} holds true with $\varphi^-=p^--1$ and $\varphi^+=p^+$.\\
        Finally, we point out that trivial computations imply that 
        $$\dfrac{d^2(\varPhi_{x,y}(\sqrt{t}))}{dt^2}\geqslant 0$$
        for all $(x,y)\in Q$ and $t\geqslant 0$. Thus, relation \hyperref[f2.]{$(\varPhi_2)$} hold true.

        Hence, we derive an existence result for  problem \hyperref[P2]{$(\mathcal{P}_{1})$} which is given by the following Remark.
       \begin{rem}\label{c2}
       If $p^--1>q^-$. 
                Then there exists $\lambda_*>0$ such that for any $\lambda\in (0,\lambda_*)$ is an eigenvalue of Problem \hyperref[P2]{$(\mathcal{P}_{1})$}.
       \end{rem}
       \begin{ex}
       	As a second example, we can take  
         $$\varphi_{x,y}(t)=\varphi_2(x,y,t)=p(x,y)\log(1+\alpha+|t|)|t|^{p(x,y)-2}t ~~\text{for all}~~ t\geqslant 0$$
         and  so,
          $$\varPhi_{x,y}(t)=\log(1+|t|)|t|^{p(x,y)}-\int_{0}^{|t|}\dfrac{\tau^{p(x,y)}}{1+\tau}d\tau,$$
          where $\alpha>0$ is a constant and   $p\in C(\overline{\Omega}\times\overline{\Omega})$ satisfies $2\leqslant p(x,y) <N$ for all $(x,y)\in  Q$.\end{ex}
          
           Then we consider the following fractional $p(x,.)$-problem  
$$
              \label{P3} (\mathcal{P}_2)~~  \left\{ 
                 \begin{array}{clclc}
              
              \left( -\Delta\right) _{\varphi_2}^{s(x,y)}u & = &  \lambda |u|^{q(x)-2}u  & \text{ in }& \Omega \\\\
                 u & = & 0 ~~  & \text{ in } & \R^N\setminus \Omega,
                 \end{array}
                 \right. 
              $$ 
              where 
              $$ \left( -\Delta\right) _{\varphi_2}^{s(x,y)}u(x)=p.v.\int_{\Omega}\dfrac{p(x,y)\log(1+\alpha+|D^{s(x,y)}u|).|D^{s(x,y)}u|^{p(x,y)-2}D^{s(x,y)}u}{|x-y|^{N+s(x,y)}}~dy~$$
               for all  $x \in \Omega.$
          
           It easy to see that $\varPhi_{x,y}$ is a Musielak function and satisfy condition \hyperref[v3]{$(\varPhi_3)$}.
           Next, we remark that   for each $(x,y)\in  Q$ fixed,   we have 
           $$p(x,y)\leqslant \dfrac{t\varphi_{x,y}(t)}{\varPhi_{x,y}(t)} ~~\text{for all}~~  t\geqslant 0.$$
     By the above information and  taking $\varphi^-=p^-$, we have 
     $$1<p^-\leqslant \dfrac{t.\varphi_{x,y}(t)}{\varPhi_{x,y}(t)} \text{  }~~\text{for all}~~(x,y)\in  Q ~~ \text{   }\text{ and all } t\geqslant 0.$$   
      On the other hand, some simple computations imply 
      $$\lim\limits_{t \rightarrow \infty }\dfrac{t.\varphi_{x,y}(t)}{\varPhi_{x,y}(t)} =p(x,y)\text{  }\text{ for all }(x,y)\in Q,$$  
      and 
      $$\lim\limits_{t \rightarrow 0 }\dfrac{t.\varphi_{x,y}(t)}{\varPhi_{x,y}(t)} =p(x,y)+1 \text{  }\text{ for all }(x,y)\in Q,$$
      Thus, we remark that $\dfrac{t.\varphi_{x,y}(t)}{\varPhi_{x,y}(t)}$ is continuous on $Q\times [0,\infty)$. Moreover, $$1<p^-\leqslant \lim\limits_{t \rightarrow 0 }\dfrac{t.\varphi_{x,y}(t)}{\varPhi_{x,y}(t)}\leqslant p^++1<\infty,$$ and $$1<p^-\leqslant \lim\limits_{t \rightarrow \infty }\dfrac{t.\varphi_{x,y}(t)}{\varPhi_{x,y}(t)}\leqslant p^++1<\infty.$$ It follows that 
      $$\varphi^+<\infty.$$
      We conclude that relation \hyperref[v1]{$(\varPhi_1)$} is satisfied.  Finally, we point out that trivial computations imply that 
          $$\dfrac{d^2(\varPhi_{x,y}(\sqrt{t}))}{dt^2}\geqslant 0$$
          for all $(x,y)\in Q$ and $t\geqslant 0$. Thus, relation \hyperref[f2.]{$(\varPhi_2)$} hold true.

     \begin{rem}\label{c3}
         If $p^->q^-$. 
    Then there exists $\lambda_*>0$ such that for any $\lambda\in (0,\lambda_*)$ is an eigenvalue of Problem \hyperref[P3]{$(\mathcal{P}_{2})$}.
        \end{rem}
        
      \section*{Disclosure statement}
                 No potential conflict of interest was reported by the authors.
                  \section*{Data Availability Statement}
                 My manuscript has no associate data.   
             


\end{document}